\documentclass{amsart}

\usepackage{amsmath,amsthm,amsfonts, amssymb,amscd}
\usepackage[all]{xy}

\newtheorem{theorem}{Theorem}[section]
\newtheorem{lemma}[theorem]{Lemma}
\newtheorem{corollary}[theorem]{Corollary}
\newtheorem{proposition}[theorem]{Proposition}

\newcommand{\minusre}{\hspace{0.3em}\raisebox{0.3ex}{\sl \tiny /}\hspace{0.3em}}
\newcommand{\minusli}{\hspace{0.3em}\raisebox{0.3ex}{\sl \tiny $\setminus $}\hspace{0.3em}}
\newcommand{\lex}{\,\overrightarrow{\times}\,}
\newcommand{\dx}{\, \mbox{\rm d}}
\newcommand{\Ker}{\mbox{\rm Ker}}
\newcommand{\Aff}{\mbox{\rm Aff}}

\begin{document}
\title{States on Pseudo Effect Algebras and Integrals}
\author{Anatolij Dvure\v censkij}
\date{}
\maketitle

\begin{center}
\footnote{


The  author thanks  for the support by Center of Excellence SAS
-~Quantum Technologies~-,  ERDF OP R\&D Projects CE QUTE ITMS
26240120009 and meta-QUTE ITMS 26240120022, the grant VEGA No.
2/0032/09 SAV. }
\small{Mathematical Institute,  Slovak Academy of Sciences\\
\v Stef\'anikova 49, SK-814 73 Bratislava, Slovakia\\
E-mail: {\tt dvurecen@mat.savba.sk}, }
\end{center}

\begin{abstract}  We show that every state on an interval pseudo effect
algebra $E$ satisfying some kind of the Riesz Decomposition
Properties (RDP) is an integral through a regular Borel probability
measure defined on the Borel $\sigma$-algebra of a Choquet simplex
$K$. In particular, if $E$ satisfies the strongest type of (RDP),
the representing Borel probability measure can be uniquely chosen to
have its support in the set of the extreme points of $K.$
\end{abstract}

\section{Introduction}

At the beginning of the Nineties, Foulis and Bennett \cite{FoBe}
introduced effect algebras with a partially defined addition, $+,$
in order to axiomatize some quantum measurements. They reflect
common features of the quantum logic $\mathcal P(H)$ of all
orthogonal projectors of a Hilbert space $H,$ that is a complete
orthomodular lattice, and  of the set of all Hermitian operators
between the operators $O$ and $I,$ $\mathcal E(H),$ that models
so-called POV-measures.

Effect algebras generalize many examples of quantum structures, like
Boolean algebras, orthomodular lattices or posets, orthoalgebras,
MV-algebras, etc. We recall that MV-algebras are algebraic
counterparts of the many-valued reasoning, and they appeared in
Mathematics under many different names, situations and motivations.
Even in the theory of effect algebras, they were defined in an
equivalent way as Phi-symmetric effect algebras, \cite{BeFo}. The
monograph \cite{DvPu} can serve as a basic source of information
about effect algebras.

Many important examples of effect algebras can be obtained as an
interval in the positive cone of a partially ordered group (=
po-group). For example, if $\mathcal B(H)$ denotes the system of all
Hermitian operators, then it is a po-group with respect to the
natural ordering of operators, and $\mathcal E(H) = [O,I] \subset
\mathcal B(H).$

During the last decade,  a whole family of interesting structures
has been appeared, like pseudo MV-algebras, \cite{GeIo}, where some
kind of commutativity was dropped. The author and Vetterlein
introduced a non-commutative version of effect algebras, called {\it
pseudo effect algebras}, see  \cite{DvVe1, DvVe2}. If a pseudo
effect algebra satisfies a stronger version of the Riesz
Decomposition Property, (RDP)$_1$, then it is also an interval in
some not necessarily commutative po-group satisfying also (RDP)$_1,$
\cite{DvVe2}. In addition, there is even a categorical equivalence
among such structures, the category  of pseudo effect algebras with
(RDP)$_1$ and the category of unital po-groups with (RDP)$_1$  that
are not necessarily Abelian.

A {\it state} is an analogue of a probability measure for quantum
structures. It is defined as a normalized additive functional on a
pseudo effect algebra preserving the partial addition $+.$ In many
cases it is connected with  a state on a unital po-group.

The state space of a pseudo effect algebra is always a compact
convex set, unfortunately, sometimes it is empty. Recently, Panti
\cite{Pan} and Kroupa \cite{Kro} proved that every state on an
MV-algebra (= Phi-symmetric effect algebra) can be represented as an
integral through a regular Borel probability measure. This result
was generalized in \cite{Dvu2} also for interval effect algebras. In
this study we continue with states on pseudo effect algebras, and
this is the main aim of the present paper. For this goal, we show
that if a pseudo effect algebra satisfies either (RDP)$_1$ or it is
an interval  in a unital po-group with (RDP), then its state space
is either empty or it is a nonempty Choquet simplex. If $E$ is even
a pseudo effect algebra with (RDP)$_2$, then its state space is
either empty or a nonempty Bauer simplex. To show that it is
necessary to study relatively bounded homomorphisms on non-Abelian
po-groups. These notions are studied in the monograph \cite[pp.
37--42]{Goo} only for Abelian po-groups. Because we are working with
po-groups that are not necessarily commutative groups, it is
necessary to exhibit these homomorphisms for our case in full
details.

Finally, this will  allow us to represent any state as a standard
integral through a regular Borel probability measure over the Borel
$\sigma$-algebra generated by the state space.

The paper is organized as follows. Section 2 is a review on pseudo
effect algebras and their basic properties. Relatively bounded
homomorphisms for not necessarily Abelian po-groups are studied in
Section 3. The state spaces of pseudo effect algebras and the
situations when they are simplices are studied in Section 4. The
main body of the paper, the integral representation of states on
pseudo effect algebras, is exhibited in Section 5.  Some final
remarks are presented in the last section.

\section{Pseudo Effect Algebras}

According to \cite{DvVe1,Dvu2}, a {\it pseudo effect algebra} is  a
partial algebra  $(E; +, 0, 1)$, where $+$ is a partial binary
operation and $0$ and $1$ are constants, such that for all $a, b, c
\in E$, the following holds

\begin{enumerate}
\item[(i)] $a+b$ and $(a+b)+c$ exist if and only if $b+c$ and
$a+(b+c)$ exist, and in this case $(a+b)+c = a+(b+c)$;

\item[(ii)]
  there is exactly one $d \in E$ and
exactly one $e \in E$ such that $a+d = e+a = 1$;

\item[(iii)]
 if $a+b$ exists, there are elements $d, e
\in E$ such that $a+b = d+a = b+e$;

\item[(iv)] if $1+a$ or $a+1$ exists, then $a = 0$.
\end{enumerate}

If we define $a \le b$ if and only if there exists an element $c\in
E$ such that $a+c =b,$ then $\le$ is a partial ordering on $E$ such
that $0 \le a \le 1$ for any $a \in E.$ It is possible to show that
$a \le b$ if and only if $b = a+c = d+a$ for some $c,d \in E$. We
write $c = a \minusre b$ and $d = b \minusli a.$ Then

$$ (b \minusli a) + a = a + (a \minusre b) = b,
$$
and we write $a^- = 1 \minusli a$ and $a^\sim = a\minusre 1$ for any
$a \in E.$

For basic properties of pseudo effect algebras see \cite{DvVe1} and
\cite{DvVe2}. We recall that if $+$ is commutative, $E$ is said to
be an {\it effect algebra}; for a comprehensive overview on effect
algebras see e.g. \cite{DvPu}.

We recall that a {\it po-group} (= partially ordered group) is a
group $G$ with a partial order, $\le,$ such that if $a\le b,$ $a,b
\in G,$ then $x+a+y \le x+b+y$ for all $x,y \in G.$  We denote by
$G^+$ the set of all positive elements of $G.$ If, in addition,
$\le$ implies that $G$ is a lattice, we call it an $\ell$-group (=
lattice ordered group). An element $u\in G^+$ is said to a {\it
strong element} if given $g \in G,$ there is an integer $n\ge 1$
such that $g \le nu,$ and the couple $(G,u)$ with a fixed strong is
called a {\it unital po-group.} The books like \cite{Fuc, Gla} can
serve as guides through the world of partially ordered groups.

For example, if $(G,u)$ is a unital (not necessary Abelian) po-group
with strong unit $u$, and
$$
\Gamma(G,u) := \{g \in G: \ 0 \le g \le u\},\eqno(2.1)
$$
then $(\Gamma(G,u); +,0,u)$ is a pseudo effect algebra if we
restrict the group addition $+$ to $\Gamma(G,u).$  Every pseudo
effect algebra $E$ that is isomorphic to some $\Gamma(G,u)$ is said
to be an {\it interval pseudo effect algebra}.

We recall that if $a^-=a^\sim$ for all $a\in E,$ ($E$ is said to be
{\it symmetric}), then $E$ is not necessarily commutative. E.g., let
$\mathbb Z$ be the group of integers and $G$ be a non-Abelian
po-group, and let $\mathbb Z \lex G$ be the lexicographical product;
$u=(0,1)$ is its strong unit. Then $E =\Gamma((\mathbb Z\lex G,
(1,0))$ is a symmetric pseudo effect algebra that is not
commutative.

According to \cite{DvVe1}, we introduce for pseudo effect algebras
the following forms of the Riesz Decomposition Properties which in
the case of commutative effect algebras can coincide:

\begin{enumerate}
 \item[(a)] For $a, b \in E$, we write $a \ \mbox{\bf com}\ b$ to mean that
for all $a_1 \leq a$ and $b_1 \leq b$, $a_1$ and $b_1$ commute.

\item[(b)] We say that $E$ fulfils the {\it Riesz
Interpolation Property}, (RIP) for short, if for any $a_1, a_2, b_1,
b_2 \in E$ such that $a_1, a_2 \,\leq\, b_1, b_2$ there is a $c \in
E$ such that $a_1, a_2 \,\leq\, c \,\leq\, b_1, b_2$.

 \item[(c)] We say that $E$ fulfils the {\it weak Riesz
Decomposition Property}, (RDP$_0$) for short, if for any $a, b_1,
b_2 \in E$ such that $a \leq b_1+b_2$ there are $d_1, d_2 \in E$
such that $d_1 \leq b_1$, $\;d_2 \leq b_2$ and $a = d_1+d_2$.

 \item[(d)] We say that $E$ fulfils the {\it Riesz
Decomposition Property}, (RDP) for short, if for any $a_1, a_2, b_1,
b_2 \in E$ such that $a_1+a_2 = b_1+b_2$ there are $d_1, d_2, d_3,
d_4 \in E$ such that $d_1+d_2 = a_1$, $\,d_3+d_4 = a_2$, $\,d_1+d_3
= b_1$, $\,d_2+d_4 = b_2$.

 \item[(e)] We say that $E$ fulfils the {\it
commutational Riesz Decomposition Property}, (RDP$_1$) for short, if
for any $a_1, a_2, b_1, b_2 \in E$ such that $a_1+a_2 = b_1+b_2$
there are $d_1, d_2, d_3, d_4 \in E$ such that (i) $d_1+d_2 = a_1$,
$\,d_3+d_4 = a_2$, $\,d_1+d_3 = b_1$, $\,d_2+d_4 = b_2$, and (ii)
$d_2 \ \mbox{\bf com}\ d_3$.

 \item[(f)] We say that $E$ fulfils the {\it strong
Riesz Decomposition Property}, (RDP$_2$) for short, if for any $a_1,
a_2, b_1, b_2 \in E$ such that $a_1+a_2 = b_1+b_2$ there are $d_1,
d_2, d_3, d_4 \in E$ such that (i) $d_1+d_2 = a_1$, $\,d_3+d_4 =
a_2$, $\,d_1+d_3 = b_1$, $\,d_2+d_4 = b_2$, and (ii) $d_2 \wedge d_3
= 0$.
\end{enumerate}

We have the implications

\centerline{ (RDP$_2$) $\Rightarrow$ (RDP$_1$) $\Rightarrow$ (RDP)
$\Rightarrow$\ (RDP$_0$) $\Rightarrow$ (RIP). }

The converse of any of these implications does not hold, see
\cite{DvVe1}. For commutative effect algebras we have

\centerline{ (RDP$_2$) $\Rightarrow$ (RDP$_1$) $\Leftrightarrow$
(RDP) $\Leftrightarrow$\ (RDP$_0$) $\Rightarrow$ (RIP). }

In addition, every pseudo effect algebra with (RDP)$_2$ is
lattice-ordered, \cite[Prop 3.3]{DvVe1}.

We note that if $E$ is an effect algebra with (RDP), then $E$ is an
interval po-group, see \cite{Rav} (\cite[Thm 1.7.17]{DvPu}), and
also if $E$ is a pseudo effect algebra with (RDP)$_1,$ then $E$ is
an interval pseudo effect algebra, see  \cite[Thm 5.7]{DvVe2}.

According to \cite{GeIo}, a {\it pseudo MV-algebra} is an algebra
$(M; \oplus,^-,^\sim,0,1)$ of type $(2,1,1,$ $0,0)$ such that the
following axioms hold for all $x,y,z \in M$,  where the derived
operation $\odot$ appearing in the axioms (A6) and (A7) is defined
by
$$ y \odot x =(x^- \oplus y^-)^\sim.$$
\begin{enumerate}
\item[{\rm (A1)}]  $x \oplus (y \oplus z) = (x \oplus y) \oplus z;$

\item[{\rm (A2)}] $x\oplus 0 = 0 \oplus x = x;$

\item[{\rm (A3)}] $x \oplus 1 = 1 \oplus x = 1;$

\item[{\rm (A4)}] $1^\sim = 0;$ $1^- = 0;$

\item[{\rm (A5)}] $(x^- \oplus y^-)^\sim = (x^\sim \oplus y^\sim)^-;$

\item[{\rm (A6)}] $x \oplus (x^\sim \odot y) = y \oplus (y^\sim
\odot x) = (x \odot y^-) \oplus y = (y \odot x^-) \oplus x;$

\item[{\rm (A7)}] $x \odot (x^- \oplus y) = (x \oplus y^\sim)
\odot y;$

\item[{\rm (A8)}] $(x^-)^\sim= x.$
\end{enumerate}

For a unital $\ell$-group $(G,u),$ set

$$\Gamma(G,u) := [0,u]$$
and
\begin{eqnarray*}
x \oplus y &:=&
(x+y) \wedge u,\\
x^- &:=& u - x,\\
x^\sim &:=& -x +u,\\
x\odot y&:= &(x-u+y)\vee 0,
\end{eqnarray*}
then $(\Gamma(G,u);\oplus, ^-,^\sim,0,u)$ is a pseudo MV-algebra,
and according to \cite{Dvu1}, for any pseudo MV-algebra there is a
unique unital $\ell$-group $(G,u)$ such that $M\cong \Gamma(G,u).$

Define $+$ to be a partial operation on $M$ that is defined for
elements $a, b \in M$ iff $a \leq b^-$, and in that case let $a+b :=
a \oplus b$. Then $(M;+,0,1)$ is a pseudo effect algebra satisfying
(RDP)$_2,$ and conversely, every pseudo effect algebra satisfying
(RDP)$_2$ can be transformed into a pseudo MV-algebra, see \cite[Thm
8.8]{DvVe2}.

We say that a mapping $s$ from a pseudo effect algebra $E$ into the
real interval $[0,1]$ is a {\it state} if $s(a+b) = s(a)+s(b)$
whenever $a+b$ is defined in $E$ and $s(1)=1.$  Let $\mathcal S(E)$
be the set of all states. Then it is a convex set, i.e. if $s_1,s_2
\in \mathcal S(E)$ and $\lambda \in [0,1],$ then $s = \lambda s_1
+(1-\lambda) s_2 \in \mathcal S(E).$ It can happen that $\mathcal
S(E)$ is empty. A state $s$ is {\it extremal} if from the property
$s = \lambda s_1 +(1-\lambda) s_2$ for some $s_1,s_2 \in \mathcal
S(E)$ and $\lambda \in (0,1),$ we conclude $s= s_1 =s_2.$ Let
$\partial_e \mathcal S(E)$ denote the set of all extremal states on
$E.$

In a similar way we define also a state on any pseudo MV-algebra.

We say that a net of states, $\{s_\alpha\},$ on $E$ {\it converges
weakly} to a state, $s,$ on $E$ if $\lim_\alpha s_\alpha (a) = s(a)$
for any $a \in E.$ Then $\mathcal S(E)$ is a compact convex
Hausdorff space, and due to the Krein--Mil'man Theorem, see
\cite[Thm 5.17]{Goo}, every state on $E$ is a weak limit of a net of
convex combinations of extremal states.

\section{Relatively Bounded Homomorphisms}

A poset $X$ is said to be {\it directed} if given $x,y \in X,$ there
is an element $z \in X$ such that $x\le z$ and $y\le z.$ It is easy
to show that a po-group $G$ is directed iff every element $g\in G$
is a difference of two elements from $G^+,$ i.e. $g = g_1-g_2 =
-g_1'+g_2',$ where $g_1,g_2, g_1', g_2' \in G^+.$

Let $(G;+,0,\le)$ be a  po-group. A subgroup $H$ of $G$ is said to
be {\it convex} if from $x\le y \le z,$ where $x,z\in H$ and $y \in
G,$ we have $y \in H.$ An {\it o-ideal} is any directed convex
subgroup of $G.$

According to \cite{DvVe1, DvVe2}, we introduce different types of
the Riesz Decomposition Properties for po-groups.

Let $(G;+,0,\le)$ be a directed po-group.

\begin{enumerate}
\item[(a)] For $a, b \in G^+$, we write $a \ \mbox{\bf com}\ b$ to mean that
for all $a_1 \leq a$ and $b_1 \leq b$,  $a_1$ and $b_1$ commute,
where $a_1,b_1 \in G^+.$

\item[(b)] We say that $G$ fulfils the {\it Riesz Interpolation
Property}, (RIP) for short, if for any $a_1, a_2, b_1, b_2 \geq 0$
such that $a_1, a_2 \,\leq\, b_1, b_2,$ there is a $c \in G$ such
that $a_1, a_2 \,\leq\, c \,\leq\, b_1, b_2$.

\item[(c)] We say that $G$ fulfils the {\it weak Riesz
Decomposition Property}, (RDP$_0$) for short, if for any $a, b_1,
b_2 \geq 0$ such that $a \leq b_1+b_2,$ there are $d_1, d_2 \in G$
such that $0 \leq d_1 \leq b_1$, $\,0 \leq d_2 \leq b_2$ and $a =
d_1+d_2$.

 \item[(d)] We say that $G$ fulfils the {\it Riesz
Decomposition Property}, (RDP) for short, if for any $a_1, a_2, b_1,
b_2 \geq 0$ such that $a_1+a_2 = b_1+b_2,$ there are $d_1, d_2, d_3,
d_4 \geq 0$ such that $d_1+d_2 = a_1$, $\,d_3+d_4 = a_2$, $\,d_1+d_3
= b_1$, $\,d_2+d_4 = b_2$.

 \item[(e)] We say that $G$ fulfils the {\it
commutational Riesz Decomposition Property}, (RDP$_1$) for short, if
for any $a_1, a_2, b_1, b_2 \geq 0$ such that $a_1+a_2 = b_1+b_2,$
there are $d_1, d_2, d_3, d_4 \geq 0$ such that (i) $d_1+d_2 = a_1$,
$\,d_3+d_4 = a_2$, $\,d_1+d_3 = b_1$, $\,d_2+d_4 = b_2$, and (ii)
$d_2 \,\mbox{{\bf com}}\, d_3$.

 \item[(f)] We say that $G$ fulfils the {\it strong
Riesz Decomposition Property}, (RDP$_2$) for short, if for any $a_1,
a_2, b_1, b_2 \geq 0$ such that $a_1+a_2 = b_1+b_2,$ there are $d_1,
d_2, d_3, d_4 \geq 0$ such that (i) $d_1+d_2 = a_1$, $\,d_3+d_4 =
a_2$, $\,d_1+d_3 = b_1$, $\,d_2+d_4 = b_2$, and (ii) $d_2 \wedge d_3
= 0$.
\end{enumerate}
Then

\centerline{ (RDP$_2$) $\Rightarrow$ (RDP$_1$) $\Rightarrow$ (RDP)
$\Rightarrow$ (RDP$_0$) $\Leftrightarrow$ (RIP), } \noindent and if
$G$ is Abelian, then {\rm (RDP$_0$)} $\Leftrightarrow$ (RDP$_1$); if
$G$ is not Abelian, the converse implications do not hold, in
general, \cite{DvVe1, DvVe2}. In addition, (RDP$_2$) holds in $G$
iff $G$ is an $\ell$-group.

Let $G$ and $H$ be po-groups. A mapping $d:G^+\to H$ is said to be
{\it subadditive} provided $d(0) = 0$ and $d(x+y)\le d(x)+d(y)$ for
all $x,y \in G^+.$

The following result is proved in \cite[Lem 2.24]{Goo} for Abelian
po-groups. An analogous proof can be used also for  po-groups with
(RIP) that are not necessarily Abelian.  Also the rest of this
section follows basic ideas of \cite[pp. 38--42]{Goo}. Because we
are interesting in general po-groups that are not studied in
\cite{Goo}, and (RIP) is not equivalent with (RDP) for non-Abelian
po-groups, we will present all our proofs in full details  for our
not necessarily commutative po-groups with (RDP).

\begin{proposition}\label{pr:3.1}  Let $G$ be a directed po-group with {\rm
(RDP)}, and let $H$ be a Dedekind complete $\ell$-group, and let
$d:G^+\to H$ be a subadditive mapping. For all $x\in G^+$, assume
that the set
$$
D(x):=\{d(x_1)+\cdots+d(x_n): x = x_1+\cdots+x_n, \ x_1,\ldots,x_n
\in G^+\} \eqno(3.1)
$$
is bounded above in $H.$  Then there is a group homomorphism $f:G\to
H$ such that $f(x)=\bigvee D(x)$ for all $x\in G^+.$

\end{proposition}

\begin{proof}  Since $H$ is a Dedekind complete $\ell$-group, due to
\cite[Cor V.20]{Fuc}, $H$ is Abelian.

Therefore, $f(x):=\bigvee D(x)$ is a well-defined mapping for all $x
\in G^+.$ It is clear that $f(0)=0$ and we are now going to show
that $f$ is additive on $G^+.$

Let $x,y \in G^+$ be given. For all decompositions
$$ x = x_1 +\cdots+x_n \ \mbox{and} \ y=y_1+\cdots +y_k$$
with all $x_i,y_j \in G^+,$ we have $x+y = x_1+\cdots+x_n +
y_1+\cdots + y_k,$ that yields

$$ \sum d(x_i)+\sum d(y_j) \le f(x+y).
$$
Therefore, $u+v \le f(x+y)$ for all $u\in D(x)$ and $b\in D(y).$
Since $H$ is Dedekind complete, $\bigvee$ is distributive with
respect to $+:$
\begin{eqnarray*}
f(x)+f(y)&=& \left(\bigvee D(x)\right) +f(y) = \bigvee_{u \in D(x)} (u+f(y))\\
&=& \bigvee_{u\in D(x)} \left(u+ \left(\bigvee D(y)\right)\right) =
\bigvee_{u\in
D(x)}\bigvee_{v \in D(y)} (u+v)\\
&\le& f(x+y).
\end{eqnarray*}

Conversely, let $x+y=z_1+\cdots+z_n,$ where each $z_i \in G^+.$ Then
(RDP) implies that there are elements $x_1,\ldots,x_n, y_1,\ldots,
y_n \in G^+$ such that $x = x_1+\cdots+x_n,$ $y = y_1+\cdots+y_n$
and $z_i = x_i+y_i$ for $i=1,\ldots,n.$  This yields
$$
\sum_i d(z_i) \le \sum_i (d(x_i)+d(y_i)) = \left(\sum_i
d(x_i)\right) + \left(\sum_i d(y_i)\right) \le f(x)+f(y),
$$
and therefore, $f(x+y)\le f(x)+f(y)$ and finally, $f(x+y)=f(x)+f(y)$
for all $x,y \in G^+.$

Since $G$ is directed and $H$ commutative, we can extend $f$ also to
the whole $G$ as follows.  The positive cone $G^+$ is a normal cone
that generates $G$,   so that every element $x\in G$ can be
expressed by $x= x_1-x_2 = - x_1'+x_2'$ for some
$x_1,x_2,x_1',x_2'\in G^+.$ Then $x_1'+x_1 = x_2'+x_2$ so that
$f(x_1)-f(x_2)=-f(x_1')+f(x_2).$ This implies that if also $x=
y_1-y_2$ for some $y_1,y_2 \in G^+,$ then
$f(x_1)-f(x_2)=f(y_1)-f(y_2),$ so that $f$ can be extended by
$f(x):= f(x_1)-f(x_2)$ whenever $x=x_1-x_2$ for $x_1,x_2 \in G^+.$
In a similar way, we have also $f(x)=-f(z_1)+f(z_2)= f(z_2)-f(z_1)$
whenever $x = -z_1+z_2$ for some $z_1,z_2 \in G^+.$

Let $a \in G^+$ and $b\in G$ be arbitrary.  Then there is an element
$a'\in G^+$ such that $a+b = b+a',$ so that $a'= -b+a+b.$  Let $b =
b_1-b_2 = -b_1+b_2,$ where $b_1,b_2,b_1',b_2'\in G^+.$ Hence,
\begin{eqnarray*}
& a+ b_1 - b_2= -b_1' +b_2'+a'\\
& b_1' +a +b_1 = b_2'+a'+b_2\\
& f(b_1') +f(a) +f(b_1) = f(b_2')+f(a')+f(b_2)\\
& f(a)+f(b)= f(b)+f(a')
\end{eqnarray*}
and the commutativity of $H$ entails $f(a)=f(a').$

Similarly, there is a unique element $a''\in G^+$ such that $a'' +b
= b+a$ and therefore, $f(a'')=f(a).$

Now we show that $f(x+y)=f(x)+f(y)$ for all $x,y \in G.$  Then $x+y
= u-v,$ $x= -x_1+x_2$ and $y = y_1- y_2$ for some $x_1,x_2,y_1,y_2,
u,v \in G^+.$ Then $u-v = -x_1 +x_2 +y_1-y_2.$  There are unique
elements $x_2',y_1'\in G^+$ such that

\begin{eqnarray*}
& u-v = -x_1+x_2 +y_1-y_2\\
& u-v = -x_2 - y_1 + x_2'+ y_1'\\
& y_2+ x_1 + u = x_2' + y_1' +v\\
& f(y_2)+ f(x_1) + f(u) = f(x_2') + f(y_1') +f(v)\\
& f(y_2)+ f(x_1) + f(u) = f(x_2) + f(y_1) +f(v)\\
& f(u)-f(v) = -f(x_1) + f(x_2)+ f(y_1) - f(y_2)\\
& f(x+y)= f(x)+f(y).
\end{eqnarray*}

This implies that $f$ is a group homomorphism.
\end{proof}

Let $X$ and $Y$ be two posets. A mapping $f:X \to Y$ is said to be
{\it relatively bounded} provided that given any subset $W$ of $X$
which is bounded (above and below) in $X$, the set $f(W)$ is bounded
in $Y.$

We recall that a group homomorphism $f$ from one po-group, $G$, into
another one, $H,$ is {\it positive} if $f(G^+)\subseteq H^+.$

\begin{proposition}\label{pr:3.2}  Let $G$ be a directed po-group
with {\rm (RDP)},  let $H$ be a Dedekind complete $\ell$-group, and
let $f:G\to H$ be a group homomorphism. Then $f$ is relatively
bounded if and only if $f=g-h$ for some positive homomorphisms $g,h:
G\to H.$
\end{proposition}

\begin{proof}
Again, $H$ is an Abelian $\ell$-group.  Assume that $f = g-h$ for
some two positive group homomorphisms $g,h: G\to H.$  If $W
\subseteq [a,b]$ in $G,$ then $g(W)\subseteq [g(a),g(b)]$ and
$h(W)\subseteq [h(a),h(b)].$  Then $g(a)-h(b) \le g(b)-h(a)$ and
$f(W)\subseteq [g(a)-h(b), g(b)-h(a)]$ that proves that $f$ is
relatively bounded.

Conversely, let $f$ be relatively bounded. If we set $d(x):=
f(x)\vee 0$ for all $x\in G^+,$ then $d(0)=0.$  For all $x,y \in
G^+$, we have
$$
d(x+y) = (f(x)+f(y))\vee 0 \le (f(x)\vee 0)+ (f(y)\vee 0) =
d(x)+d(y),
$$
so that $d$ is subadditive.

Let us define $D(x)$ by (3.1) for each $x \in G^+$. We assert that
$D(x)$ is bounded above in $H.$  By the assumption, there are
elements $a,b \in H$ such that $f([0,x])\subseteq [a,b].$  Fix a
decomposition $x = x_1+\cdots+x_n$ with $x_i\in G^+$ for each
$i=1,\ldots, n.$ By \cite[Lem 1.21]{Goo}, we have

$$ \sum_{i=1}^nd(x_i) =\sum_{i=1}^n (f(x_i)\vee 0) =
\left(\bigvee_{A \in 2^n}\left(\sum_{i\in A} f(x_i)\right)\right)
\vee 0.
$$
For all $A \in 2^n,$ we have
$$ 0 \le \sum_{i\in A}x_i \le x,
\mbox{and}\ \sum_{i\in A} f(x_i) = f(\sum_{i\in A}x_i) \le b.
$$
Hence, $d(x_1) + \cdots+d(x_n) \le b\vee 0,$ and consequently,
$b\vee 0$ is an upper bound for $D(x)$ that proves the assertion.

By Proposition \ref{pr:3.1}, there exists a group homomorphism
$g:G\to H$ such that $g(x) = \bigvee D(x)$ for all $x \in G^+.$
Since $g(x)\ge d(x) \ge 0,$ $g(x)$ is a positive homomorphism, and
$g(x)\ge d(x) \ge f(x)$ for all $x \in G^+.$ Hence, $ h = g-f$ is a
positive homomorphism, too.
\end{proof}

Let $G$ be a directed po-group and $H$ an Abelian po-group.  The
set, $\mbox{\rm Hom}(G,H),$ of all group homomorphisms from $G$ into
$H$ is an Abelian group. Given $f,g \in \mbox{\rm Hom}(G,H)$, we
define $f\le^+ g$ whenever $g-f$ is a positive group homomorphism.
Then $\le^+$ is a partial order and $\mbox{\rm Hom}(G,H)$ is  an
Abelian po-group with respect to this partial order. Indeed, it is
easy to see that $\le^+$ is a preorder.  Assume now $(f-g)(x)\ge 0$
and $(g-f)(x)$ for all $x \in G^+,$ and thus $f(x)=g(x)$ for all $x
\in G^+.$ The group $G$ is directed and $G^+$ is a normal cone of
$G$ that  generates $G$ as a group. If $x=x_1-x_2$ with $x_1,x_2\in
G^+,$ then $f(x)=f(x_1)-f(x_2)=g(x_1)-g(x_2)=g(x),$ we see that
$f(x)=g(x)$ for all $x \in G.$

It is now clear that  the positive cone of $\mbox{Hom}(G,H)$
consists of all positive homomorphisms from $G$ into $H.$ It is
non-void because the zero homomorphism from $G$ into $H$ belongs to
it.

\begin{proposition}\label{pr:3.3} Let $G$ be a directed po-group
with {\rm (RDP)},  let $H$ be a Dedekind complete $\ell$-group, and
let $B(G,H)$ be the set of all relatively bounded group
homomorphisms from $G$ to $H.$  Then $B$ is a nonempty o-ideal of
$\mbox{\rm Hom}(G,H).$
\end{proposition}

\begin{proof} Because the zero homomorphism from $G$ to $H$ is a
relatively bounded group homomorphism, $B(G,H)$ is non-void, and
according to Proposition \ref{pr:3.2}, $B(G,H)$ equals  the subgroup
of $\mbox{\rm Hom}(G,H)$ generated by the positive homomorphisms.
Therefore, $B(G,H)$ is a directed subgroup of $\mbox{\rm Hom}(G,H).$

Given $f \in \mbox{\rm Hom}(G,H)$ and $g \in B(G,H)$ such that
$0\le^+ f \le^+ g,$ write $g=g_1-g_2$ for some positive
homomorphisms $g_1,g_2 \in \mbox{\rm Hom}(G,H).$ Since $f\le^+ g\le
^+  g_1,$ we have $f=g_1 - (g_1-f)$ with $g_1$ and $g_1-f$ positive
homomorphisms, and hence, $f \in B(G,H).$ This proves that $B(G,H)$
is an o-ideal.
\end{proof}

\begin{theorem}\label{th:3.4} Let $G$ be a directed po-group
with {\rm (RDP)} and  let $H$ be a Dedekind complete $\ell$-group.

\begin{enumerate}

\item[(a)] The group $B(G,H)$ of all relatively bounded group
homomorphisms from $G$ to $H$ is a Dedekind complete $\ell$-group.

\item[(b)] If $\{f_i\}_{i\in I}$ is a nonempty system of $B(G,H)$
that is bounded above, and if $d(x)=\bigvee_i f_i(x)$ for all $x \in
G^+,$ then
$$ \left(\bigvee_i f_i\right)(x) = \bigvee\{d(x_1)+\cdots + d(x_n):
x= x_1+\cdots + x_n, \ x_1,\ldots, x_n \in G^+\}
$$
for all $x \in G^+.$

\item[(c)] If $\{f_i\}_{i\in I}$ is a nonempty system of $B(G,H)$
that is bounded below, and if $e(x)=\bigwedge_i f_i(x)$ for all $x
\in G^+,$ then
$$ \left(\bigwedge_i f_i\right)(x) = \bigwedge\{e(x_1)+\cdots + e(x_n):
x= x_1+\cdots + x_n, \ x_1,\ldots, x_n \in G^+\}
$$
for all $x \in G^+.$

\end{enumerate}
\end{theorem}

\begin{proof} Let $g \in B(G,H)$ be an upper bound for $\{f_i\}.$
For any $x \in G^+$, we have $f_i(x)\le g(x),$ so that the mapping
$d(x)=\bigvee_i f_i(x)$ defined on $G^+$ is a a subadditive mapping.
For any $x \in G^+$ and any decomposition $x = x_1+\cdots + x_n$
with all $x_i \in G^+,$ we conclude $d(x_1)+\cdots+ d(x_n)\le
g(x_1)+\cdots + g(x_n)=g(x).$  Hence, $g(x)$ is an upper set for
$D(x)$ defined by (3.1).

Proposition \ref{pr:3.1} entails there is a group homomorphism $f:G
\to H$ such that $f(x)=\bigvee D(x).$ For every $x \in G^+$ and
every $f_i$ we have $f_i(x)\le d(x)\le f(x)$ that gives $f_i \le^+
f.$ The mappings $f-f_i$ are positive homomorphisms belonging bo
$B(G,H)$ that gives $f \in B(G,H).$ If $h \in B(G,H)$ such that
$f_i\le^+ h$ for any $i\in I,$ then $d(x)\le h(x)$ for any $x \in
G^+.$ As above, we can show that $h(x)$ is also an upper bound for
$D(x)$, whence $f(x)\le h(x)$ for any $x \in G^+$ that gives $f\le
^+h.$ In other words, we have proved that $f$ is the supremum of
$\{f_i\}_{i\in I},$ and its form is given by (b).

Now if we apply  the order anti-automorphism $z\mapsto - z$ in $H$,
we see that infima exist in $B(G,H)$ for any bounded below system
$\{f_i\}_{i\in I},$ and their form is given by (c).

By Proposition \ref{pr:3.2}, $B(G,H)$ is directed, combining (b) and
(c), we see that $B(G,H)$ is a Dedekind complete $\ell$-group.
\end{proof}

If $H=\mathbb R,$ Theorem \ref{th:3.4} can be reformulated as
follows.

\begin{theorem}\label{th:3.5} If $G$ is a directed po-group with
{\rm (RDP)}, then the group $B(G,\mathbb R)$ of all relatively
bounded group homomorphisms from $G$ to $\mathbb R$ is a Dedekind
complete lattice ordered real vector space. Given $f_1,\ldots, f_n
\in B(G,\mathbb R)$,

\begin{eqnarray*}
\left( \bigvee_{i=1}^n f_i\right)(x) = \sup\{f_1(x_1)+\cdots
+f_n(x_n):
x = x_1+\cdots +x_n,\ x_1,\ldots, x_n \in G^+\},\\
\left( \bigwedge_{i=1}^n f_i \right)(x) = \inf\{f_1(x_1)+\cdots
+f_n(x_n):
x = x_1+\cdots +x_n,\ x_1,\ldots, x_n \in G^+\},\\
\end{eqnarray*}
for all $x \in G^+.$
\end{theorem}

\begin{proof} Due to Theorem \ref{th:3.4}, $B(G,\mathbb R)$ is a
Dedekind complete $\ell$-group. It is evident that it is a Riesz
space, i.e., a lattice ordered real vector space.

Take $f_1,\ldots,f_n \in B(G,\mathbb R)$ and let $f = f_1\vee \cdots
\vee f_n.$ For any $x \in G^+$ and $x=x_1+\cdots+x_n$ with
$x_1,\ldots, x_n \in G^+,$ we have  $f_1(x_1)+\cdots + f_n(x_n) \le
f(x_1)+\cdots + f(x_n) = f(x).$  Due to Theorem \ref{th:3.4}, given
an arbitrary real number $\epsilon >0$, there is a decomposition $x
= y_1+\cdots+y_k$ with $y_1,\ldots,y_k \in G^+$ such that

$$ \sum_{j=1}^k \max\{f_1(y_j),\ldots,f_n(y_j)\} > f(x)-\epsilon.
$$

We note that if $a \in G^+$ and $b\in G,$ the elements $a',a''\in
G^+$ such that $a+b= b+a'$ and $b+a =a''+b$ are said to be (right
and left) conjugates of $a$ by $b$. Since $\mathbb R$ is Abelian,
for any $h\in B(G,\mathbb R),$ $h(a')=h(a)=h(a'').$

If $k < n$, we can add the zero element to the decomposition, if
necessary, so that without loss of generality, we can assume that
$k\ge n.$

We decompose the set $\{1,\ldots,k\}$ into mutually disjoint sets
$J(1),\ldots,J(n)$ such that
$$J(i):=\{j \in \{1,\ldots,k\}: \max\{f_1(y_j),\ldots, f_n(y_j)\}=
f_i(y_j)\}.
$$
Assume $J(1)=\{j_{t_1},\ldots, j_{n_1}\}.$ Since $G^+$ is a normal
cone of $G$, $x$ can be expressed in the form $x=x_{j_{t_1}}+\cdots+
x_{j_{n_1}} + x_j' +\cdots + x_k',$ where $x_j',\ldots, x_k'\in G^+$
are conjugates of $x_j,\ldots,x_k.$

Let $x_1:=x_{j_{t_1}}+\cdots+ x_{j_{n_1}}$.

In a similar way, let $J(2)=\{j_{t_2},\ldots, j_{n_2}\}$ and let
$x_2 = y_{j_{t_2}}+\cdots +y_{j_{n_2}}.$ Again, we can express $x$
in the form $x = x_1 + x_2 + y_s'' +\cdots + y_k'',$ where $y_t''$'s
are appropriate  conjugates of $y_s',\ldots, y_k'.$ Processing in
this way for each $J(i) = \{j_{t_i},\ldots, j_{n_i}\},$ we define
the element $x_i = c_{t_{j_{t_i}}} +\cdots + c_{t_{j_{n_i}}},$ where
$c_{t_{j_s}}$ is an appropriate conjugate of the element
$y_{t_{j_s}}.$  Then $x = x_1+\cdots + x_n,$ and

$$ \sum_{i=1}^n f_i(x_i) = \sum_{i=1}^n \sum_{j \in
J(i)}f_i(y_j)=\sum_{i=1}^k\max\{f_1(y_j),\ldots, f_n(y_j)\} >
f(x)-\epsilon.$$

This implies $f(x)$ equals the given supremum.

The formula for $(f_1\wedge \cdots \wedge f_n)(x)$ can be obtained
applying   the order anti-automorphism $z\mapsto -z$ holding in
$\mathbb R.$
\end{proof}

\section{State Spaces of Pseudo Effect Algebras and Simplices}

Simplices are important mathematical tools that can be used also for
analysis of the state space of a pseudo effect algebra. In
particular, we show that, for any pseudo effect algebra $E$ with
(RDP)$_1,$ the state space of $E$ is either an empty set or it is a
nonempty simplex. This result generalizes analogous result holding
for effect algebras, see \cite[Thm 5.1]{Dvu}.

Now we present some elements of simplices. For a good source about
convex sets, see the monographs \cite{Alf, Phe, Goo, AlSc}.

Let $K_1, K_2$ be two convex sets. A mapping $f:K_1\to K_2$ is said
to be {\it affine} if it preserves all convex combinations, and if
$f$ is also injective and surjective such that also $f^{-1}$ is
affine, $f$ is an {\it affine isomorphism} and $K_1$ and $K_2$ are
{\it affinely isomorphic}.

We recall that a {\it convex cone} in a real linear space $V$ is any
subset $C$ of  $V$ such that (i) $0\in C,$ (ii) if $x_1,x_2 \in C,$
then $\alpha_1x_1 +\alpha_2 x_2 \in C$ for any $\alpha_1,\alpha_2
\in \mathbb R^+.$  A {\it strict cone} is any convex cone $C$ such
that $C\cap -C =\{0\},$ where $-C=\{-x:\ x \in C\}.$ A {\it base}
for a convex cone $C$ is any convex subset $K$ of $C$ such that
every non-zero element $y \in C$ may be uniquely expressed in the
form $y = \alpha x$ for some $\alpha \in \mathbb R^+$ and some $x
\in K.$

We recall that in view of \cite[Prop 10.2]{Goo}, if $K$ is a
non-void convex subset of $V,$ and if we set

$$ C =\{\alpha x:\ \alpha \in \mathbb R^+,\ x \in K\},
$$
then $C$ is a convex cone in $V,$ and $K$ is a base for $C$ iff
there is a linear functional $f$ on $V$ such that $f(K) = 1$ iff $K$
is contained in a hyperplane in $V$ which misses the origin.

Any strict cone $C$ of $V$ defines a partial order $\le_C$ via $x
\le_C y$ iff $y-x \in C.$ It is clear that $C=\{x \in V:\ 0 \le_C
x\}.$ A {\it lattice cone} is any strict convex cone $C$ in $V$ such
that $C$ is a lattice under $\le_C.$

A {\it simplex} in a linear space $V$ is any convex subset $K$ of
$V$ that is affinely isomorphic to a base for a lattice cone in some
real linear space. A  simplex $K$ in a locally convex Hausdorff
space is said to be (i) {\it Choquet} if $K$ is compact, and (ii)
{\it Bauer} if $K$ and $\partial_e K$ are compact, where $\partial_e
K$ is the set of extreme points of $K.$

For example, for the important quantum mechanical  example, if $H$
is a separable complex Hilbert space, ${\mathcal S}({\mathcal
E}(H))$ is not a simplex due to \cite[Thm 4.4]{AlSc} or \cite[Ex
4.2.6]{BrRo}, where $\mathcal E(H)$ is the system of all Hermitian
operators on a Hilbert space that are between the zero operator and
the identity operator. On the other hand, the state space of a
commutative C$^*$-algebra and the trace space of a general
C$^*$-algebra are Choquet simplices, see \cite[p. 7]{AlSc} or
\cite[Ex 4.2.6]{BrRo}.

Let $(G,u)$ be a unital  po-group with strong unit. A {\it state} on
$(G,u)$ we understand any positive homomorphism $s: G \to \mathbb R$
that is normalized, i.e. $s(u)=1.$  Let $\mathcal S(G,u)$ be the set
of all states on $(G,u).$  If $(G,u)$ is an Abelian po-group, due to
\cite[Cor 4.3]{Goo}, $\mathcal S(G,u)$ is always nonempty, whenever
$G\ne \{0\}.$ This is not true, in general, for non-Abelian unital
po-groups, even not for unital $\ell$-groups, see \cite[Cor
4.7]{Dvu}. On the other hand, if $(G,u)$ is a linearly ordered, then
$\mathcal S(G,u)$ is a singleton, \cite[Thm 5.6]{Dvu}. It is
possible to show \cite[Prop 4.3, 4.6]{Dvu} that for a unital
$\ell$-group $(G,u),$ a state $s$ is extremal iff $\Ker(s):=\{g\in
G: s(|g|)=0\}$ is a maximal $\ell$-ideal (= lattice ordered o-ideal)
that is normal. Therefore, $\mathcal (G,u)$ is non-void iff the
unital $\ell$-group $(G,u)$ has at least one maximal $\ell$-ideal
that is also normal.

In a similar way as for pseudo effect algebras, we can define an
extremal state, and let $\partial_e \mathcal S(G,u)$ be the set of
extremal states on $(G,u)$.  We say that a net of states,
$\{s_\alpha\},$ on $(G,u)$ {\it converges weakly} to a state, $s,$
if $\lim_\alpha s_\alpha (g) = s(g)$ for any $g \in G.$  Then
$\mathcal S(G,u)$ is a compact convex Hausdorff space, and due to
the Krein--Mil'man Theorem, see \cite[Thm 5.17]{Goo}, every state on
$(G,u)$ is a weak limit of a net of convex combinations of extremal
states.

\begin{proposition}\label{pr:4.1}  Let $E=\Gamma(G,u),$ where
$(G,u)$ is a unital po-group satisfying {\rm (RDP)}. Then the
restriction of every state $s$ on $E$ gives a state on $E,$ and
conversely, every state on $E$ can be uniquely extended to a state
on $(G,u).$  Moreover, the state spaces $\mathcal S(G,u)$ and
$\mathcal S(E)$ are affinelly homeomorphic.
\end{proposition}

\begin{proof} It is evident that the restriction of any state on
$(G,u)$ is a state on $E$.  Conversely, let $s$ be a state on $E$.
We extend $s$ onto a state $\hat s$ defined on $G^+$ via $\hat s(x)
= s(x_1)+\cdots+ s(x_n)$ whenever $x = x_1+\cdots+ x_n,$ where
$x_1,\ldots, x_n \in E.$  We show that $\hat s$ is well-defined,
indeed, let $x = y_1+\cdots+y_m$ with $y_1,\ldots, y_m\in E.$ The
(RDP) entails that there is a finite system $\{c_{ij}: i=1,\ldots,n
,\ j = 1,\ldots, m\}$ from $G^+$ such that every $x_i = \sum_{j=1}^m
c_{ij}.$  This implies each $c_{ij}$ is from $E.$ Check:
$\sum_{i=1}^n s(x_i) = \sum_{i=1}^n \sum_{j=1}^m s(c_{ij}) =
\sum_{j=1}^m s(y_j).$

Since $\hat s$ is additive on $G^+$ and $G^+$ generates $G,$ $\hat
s$ can be easily extended to a unique state on the whole $G.$

Therefore, the mapping $s\mapsto \hat s$ defined on $\mathcal S(E)$
is injective, surjective and continuous. If $s \in \partial_e
\mathcal S(E),$ then $\hat s \in \partial_e \mathcal S(G,u),$ and
vice versa.

Therefore, the state spaces $\mathcal S(E)$ and $\mathcal S(G,u)$
are affinelly homeomorphic.
\end{proof}

\begin{theorem}\label{th:4.2}  If $(G,u)$ is a unital po-group with
{\rm (RDP)}, then either $\mathcal S(G,u)$ is empty or it is a
nonempty  Choquet simplex. In addition, the same is true for
$\mathcal S(\Gamma(G,u)).$
\end{theorem}

\begin{proof} Assume that $\mathcal S(G,u)$ is nonempty. Then the
positive cone  $B(G,\mathbb R)^+$ of the Abelian Dedekind complete
$\ell$-group $B(G,\mathbb R)$ consists of all positive homomorphisms
from $G$ into $\mathbb R,$ so that $B(G,\mathbb R)^+ =\{\alpha s:
\alpha \in \mathbb R^+, s \in \mathcal S(G,u)\}.$  The set $\mathcal
S(G,u)$ lies in the hyperplane $\{f \in B(G,\mathbb R): f(u)= 1\}$
which misses the origin. Therefore, $\mathcal S(G,u)$ is a base for
$B(G,\mathbb R)^+,$ and $\mathcal S(G,u)$ is a simplex.  Since
$\mathcal S(G,u)$ is compact, $\mathcal S(G,u)$ is a Choquet
simplex.

Since $\mathcal S(G,u)$ and $\mathcal S(\Gamma(G,u))$ are affinely
homeomorphic, Proposition \ref{pr:4.1}, we conclude  that $\mathcal
S(\Gamma(G,u))$ is also a Choquet simplex.
\end{proof}

\begin{theorem}\label{th:4.3}  Let $E$ be a pseudo effect algebra
with {\rm (RDP)}$_1.$ Then $\mathcal S(E)$ is either an empty set or
it is a nonempty Choquet simplex.
\end{theorem}

\begin{proof}  Let $E$ be a pseudo effect algebra with (RDP)$_1$ and let
$\mathcal S(E)\ne \emptyset.$  Due to the basic representation of
pseudo effect algebras with (RDP)$_1$, \cite[Thm 5.7]{DvVe2}, there
is a unique (up to isomorphism) unital po-group $(G,u)$ with
(RDP)$_1$ such that $E \cong \Gamma(G,u).$  Because $(G,u)$
satisfies also (RDP), by Theorem \ref{th:4.2}, $\mathcal
S(\Gamma(G,u))$ and $\mathcal S(E)$ are affinelly isomorphic Choquet
simplices.
\end{proof}

\begin{theorem}\label{th:4.4}  Let $E$ be a pseudo effect algebra
with {\rm (RDP)}$_2.$ Then $\mathcal S(E)$ is either an empty set or
it is a nonempty Bauer simplex.
\end{theorem}

\begin{proof} By \cite[Thm 5.7]{DvVe2}, there is a unique unital po-group
$(G,u)$ with {\rm (RDP)}$_1$ such that $E \cong \Gamma(G,u).$ By
\cite[Prop. 6.3]{DvVe2}, $(G,u)$ satisfies even {\rm (RDP)}$_2.$

Assume that $\mathcal S(E)\ne \emptyset.$ As it was already
mentioned  at the end of Section 2, $E$ can be converted into a
pseudo MV-algebra. Due to \cite[Prop 4.7]{Dvu}, a state $s$ on a
pseudo MV-algebra is extremal iff $s(a\wedge b)=\min\{s(a),s(b)\}$
for all $a,b \in E.$ Therefore, $\partial_e \mathcal S(E)$ is
compact, and by Theorem \ref{th:4.3}, $\mathcal S(E)$ is a compact
simplex. Hence, $\mathcal S(E)$ is a nonempty Choquet simplex.
\end{proof}

\section{Representation of States by Integrals}

In this main section of the paper, we show that if $s$ is a state on
a pseudo effect algebra with (RDP)$_1,$ then it can be represented
as an integral of a continuous affine function through some regular
Borel probability measure.  It will generalize analogous  results
from \cite{Dvu2}.

We start with some necessary definitions.

Let $K$ be a compact convex subset of a locally convex Hausdorff
space. A mapping $f:\ K \to \mathbb R$ is said to be {\it affine}
if, for all $x,y \in K$ and any $\lambda \in [0,1]$, we have
$f(\lambda x +(1-\lambda )y) = \lambda f(x) +(1-\lambda ) f(y)$. Let
$\Aff(K)$ be the set of all continuous affine functions on $K.$ Then
$\mbox{Aff}(K)$ is a unital po-group with the strong unit $1$ which
is a subgroup  of the po-group $\mbox{C}(K)$ of all continuous
real-valued functions on $K$ (we recall that, for $f,g \in
\mbox{C}(K),$ $f \le g$ iff $f(x)\le g(x)$ for any $x \in K$), hence
it is an Archimedean unital po-group with the strong unit $1$ that
is even an $\ell$-group.

For example, let $E$ be a pseudo effect algebra such that $\mathcal
S(E)\ne \emptyset.$ Given $a \in E,$ let $\hat a:\mathcal S(E) \to
[0,1]$ such that $\hat a(s):= s(a),$ $s \in \mathcal S(E).$  Then
$\hat a \in \Aff(\mathcal S(E)).$ In a similar way, if $\mathcal
S(G,u)\ne \emptyset,$  for any $g \in (G,u),$ the mapping $\hat g:
\mathcal S(G,u) \to \mathbb R$ defined by $\hat g(s):=s(g),$ $s \in
\mathcal S(G,u),$ is a continuous affine function on $\mathcal
S(G,u).$

Let $S = {\mathcal S}(\mbox{Aff}(K),1).$ Then the evaluation mapping
$\psi:\ K \to S$ defined by $\psi(x)(f)=f(x)$ for all $f \in
\mbox{Aff}(K)$ $(x \in K)$ is an affine homeomorphism of $K$ onto
$S,$ see \cite[Thm 7.1]{Goo}.

The po-group $\mbox{Aff}(K)$ is not necessarily neither with (RIP)
nor an $\ell$-group. By \cite[Thm 11.4]{Goo}, $\Aff(K)$ has (RIP)
iff $K$ is a Choquet simplex, and \cite[Thm 11.21]{Goo}, $\Aff(K)$
is an $\ell$-group iff $K$ is a Bauer simplex. Therefore,  due to
Theorems \ref{th:4.2}-\ref{th:4.4} we have the following result:

\begin{theorem}\label{th:5.1}
Let for a unital po-group $(G,u),$ the state space $\mathcal S(G,u)$
be non-void. If $(G,u)$ has {\rm (RDP)}, then $(\Aff(\mathcal
S(G,u)),1)$ is an Abelian unital po-group with {\rm (RIP)}.

If $(G,u)$ has {\rm (RDP)$_2$}, then $(\Aff(\mathcal S(G,u)),1)$ is
an Abelian unital po-group with {\rm (RDP)$_2$}.

Let $E$ be a pseudo effect algebra admitting at least one state. If
$E$ has {\rm (RDP)$_1$}, then $(\Aff(\mathcal S(E)),1)$ is an
Abelian  unital po-group with {\rm (RIP)}, if $E$ has {\rm
(RDP)$_2$}, then $(\Aff(\mathcal S(E)),1)$ is an Abelian unital
$\ell$-group.
\end{theorem}

If  $K$ is a compact Hausdorff topological space, let ${\mathcal
B}(K)$ be the Borel $\sigma$-algebra of $K$ generated by all open
subsets of $K.$  Let  ${\mathcal M}_1^+(K)$ denote the set of  all
probability measures, that is, all positive regular
$\sigma$-additive Borel measures $\mu$ on $\mathcal B(K).$  We
recall that a Borel measure $\mu$ is called {\it regular} if

$$\inf\{\mu(O):\ Y \subseteq O,\ O\ \mbox{open}\}=\mu(Y)
=\sup\{\mu(C):\ C \subseteq Y,\ C\ \mbox{closed}\}
$$
for any $Y \in {\mathcal B}(K).$

For example, if $x \in K,$ then the Dirac measure $\delta_x$
concentrated at the point $x$ is a regular Borel probability measure
on $\mathcal B(K).$

For two measures $\mu$ and $\lambda$ we write
$$\mu \sim \lambda\quad  \mbox{iff}\quad
\int_K f \dx \mu=\int_K f \dx \lambda,\ f \in \mbox{Aff}(K).
$$

If  $\mu $ and $\lambda$ are nonnegative regular Borel measures on a
convex compact set $K,$  we introduce for them the {\it Choquet
ordering} defined by
$$
\mu \prec \lambda  \quad  \mbox{iff} \quad \int_K f \dx \mu\le\int_K
f \dx \lambda, \ f \in \mbox{Con}(K),
$$  where $\mbox{Con}(K)$ is the set of all continuous convex
functions $f$ on $K$ (that is $f(\alpha x_1+(1-\alpha) x_2)\le
\alpha f(x_1)+(1-\alpha)f(x_2)$ for $x_1,x_2\in K$ and $\alpha \in
[0,1]$). Then $\prec$ is a partial order on the cone of nonnegative
measures. The fact $\lambda \prec \mu$ and $\mu \prec \lambda$
implies $\lambda = \mu$ follows from the fact that
$\mbox{Con}(K)-\mbox{Con}(K)$ is dense in $\mbox{\rm C}(K).$

Moreover, for any probability measure (= regular Borel probability
measure) $\lambda$ there is a maximal probability measure $\mu$ such
that $\mu \succ \lambda,$ \cite[Lem 4.1]{Phe}.

The fact $\mu \succ \lambda$ means that $\lambda$ has its support
``closer" to the extreme points of $K$ than does $\lambda.$

The following results have been proved in \cite{Dvu2} for interval
effect algebras.  Here we generalize them  for pseudo effect
algebras with (RDP), (RDP)$_1,$ and (RDP)$_2,$ respectively.
However, the situation for pseudo effect algebras follows basic
ideas of the analogous proofs from \cite{Dvu2}, we present the
proofs here in  full generality because it was necessary to take
into account a non-commutative character of pseudo effect algebras
that was developed in the previous sections.

\begin{theorem}\label{th:7.2'}
Let $E=\Gamma(G,u)$ be a pseudo effect algebra such that it admits
at least one state, where $(G,u)$ is a unital po-group with {\rm
(RDP)} and let $s$ be a state on $E.$ Let $\psi: E \to \Aff(\mathcal
S(E))$ be defined by $\psi(a) := \hat a,$ $a\in E,$ where $\hat a$
is a mapping from $\mathcal S(E)$ into $[0,1]$ such that $\hat
a(s):=s(a),$ $s \in \mathcal S(E).$ Then there is a unique state
$\tilde s$ on the unital Abelian po-group $(\Aff(\mathcal S(E)),1)$
such that $\tilde s(\hat a) = s(a)$ for any $a \in E.$

The mapping $s \mapsto \tilde s$ defines an affine homeomorphism
from the state space $\mathcal S(E)$ onto $\mathcal
S(\Gamma(\Aff(\mathcal S(E)),1)).$

\end{theorem}

\begin{proof}
By Proposition \ref{pr:4.1}, the mapping $\psi$ can be uniquely
extended to a po-group homomorphism $\hat \psi:G \to \Aff(\mathcal
S(E))$ via $\hat \psi(g)(s):= s(g),$ $s \in \mathcal S(G,u).$ Let
$\hat s$ be a state on $(G,u)$ that is a unique extension of a state
$s.$  Now applying  the proof of  \cite[Prop 7.20]{Goo}, we can show
that \cite[Prop 7.20]{Goo} holds also for our group $G$ that is not
necessarily bounded. Therefore,  we have that there is a unique
state $\tilde s$  on $(\Aff(\mathcal S(E)),1)$ such that $\tilde
s(\hat a) = s(a),$ $a\in E.$


The affine homeomorphism $s \mapsto \tilde s$ follows from \cite[Thm
7.1]{Goo}.
\end{proof}

\begin{theorem}\label{th:7.3'}  Let $(G,u)$ be a unital po-group
with {\rm (RDP)} and let $s$ be a state on the pseudo effect algebra
$E=\Gamma(G,u).$ Then there is a unique maximal regular Borel
probability measure $\mu_s \sim \delta_s$ on $\mathcal B(\mathcal
S(E))$ such that

$$ s(a) = \int_{\mathcal S(E)} \hat a(x) \dx \mu_s(x),\quad a \in
E. \eqno(5.1)$$
\end{theorem}

\begin{proof}
Due to Theorem \ref{th:4.2}, $\mathcal S(E)$ is a Choquet simplex.
By Theorem \ref{th:7.2'}, there is a unique state $\tilde s$ on
$(\Aff(\mathcal S(E)),1)$ such that $\tilde s(\hat a) = s(a),$ $a
\in A.$

Applying the Choquet--Meyer Theorem, \cite[Thm p. 66]{Phe},  we have

$$f(s)=\int_{\mathcal S(E)} f(x) \dx\mu_s, \quad f \in \Aff(\mathcal
S(E)).
$$
Since $\hat a \in \Aff(\mathcal S(E))$ for any $a\in E,$ we have the
representation given by (5.1).
\end{proof}

\begin{theorem}\label{th:7.3'''} Let $E$ be a pseudo effect algebra
satisfying {\rm (RDP$_1$)} and let $s$ be a state on $E.$ Then there
is a unique maximal regular Borel probability measure $\mu_s \sim
\delta_s$ on $\mathcal B(\mathcal S(E))$ such that

$$ s(a) = \int_{\mathcal S(E)} \hat a(x) \dx \mu_s(x),\quad a \in
E.$$
\end{theorem}

\begin{proof}  By \cite[Thm 5.7]{Goo}, there is a unital po-group
$(G,u)$ with (RDP)$_1$ such that $E \cong \Gamma(G,u).$  The desired
result follows now from Theorem \ref{th:7.3'}.
\end{proof}

\begin{theorem}\label{th:7.5'}  Let $E$ be a pseudo effect algebra with
{\rm (RDP)$_2$} and let $s$ be a state on $E.$ Then there is a
unique regular Borel probability measure, $\mu_s,$ on $\mathcal
B(\mathcal S(E))$ such that $\mu_s(\partial_e \mathcal S(E))=1$ and

$$ s(a) = \int_{\partial_e \mathcal S(E)} \hat a(x) \dx \mu_s(x),\quad a \in
E. \eqno(5.2)$$
\end{theorem}

\begin{proof}  Due to Theorem \ref{th:7.3'}, we have a unique regular
Borel probability measure $\mu_s\sim \delta_s$ such that (5.1)
holds. The characterization of Bauer simplices, \cite[Thm
II.4.1]{Alf}, says that then $\mu_s$ is a unique regular Borel
probability measure $\mu_s$ on such that (5.1) holds and
$\mu_s(\partial_e \mathcal S(E)) = 1.$ Hence, (5.2) holds.
\end{proof}

\begin{corollary}\label{co:5.6} Let $s$ be a state on a pseudo
MV-algebra $E.$  Then there is a unique regular Borel probability
measure, $\mu_s,$ on $\mathcal B(\mathcal S(E))$ such that
$\mu_s(\partial_e \mathcal S(E))=1$ and

$$ s(a) = \int_{\partial_e \mathcal S(E)} \hat a(x) \dx \mu_s(x),\quad a \in
E.$$
\end{corollary}

\begin{proof}  From \cite[Thm 8.8]{DvVe2}, we have that the pseudo MV-algebra
$E$ can be converted into a pseudo effect algebra with (RDP)$_2.$
Since the notions of states on pseudo MV-algebras and on pseudo
effect algebras coincide, the corollary follows from Theorem
\ref{th:7.5'}.
\end{proof}

Corollary \ref{co:5.6} generalizes the analogous statements proved
in \cite{Kro, Pan} for MV-algebras.

We endow the set of regular Borel probability measures with the
weak$^*$  topology, i.e., a net $\{\mu_\alpha\}$ converges to an
element $\mu $ iff $\int_K f\dx \mu_\alpha \to \int_K f\dx\mu$ for
all $f \in \mbox{\rm C}(K)$.

Any convex subset $F$ of a convex set $K$ is a {\it face} if
$x=\lambda x_1+(1-\lambda)x_2 \in F,$ $0<\lambda<1,$ entail
$x_1,x_2\in F.$

\begin{corollary}\label{co:7.6'} Let $E$ be a  pseudo effect algebra with
{\rm (RDP)$_2$} and let $\mathcal S(E)$ be nonempty. Let $F$ be the
set of regular Borel probability measures $\mu \in \mathcal
M_1^+(\mathcal S(E))$   such that $\mu(\partial_e \mathcal S(E))=1$
is a closed face. The mapping $s \mapsto \mu_s,$ where $\mu_s$ is a
unique regular Borel probability measure satisfying {\rm (5.2)} and
$\mu_s(\partial_e \mathcal S(E))=1$, is an affine homeomorphism
between $\mathcal S(E)$ and $F$ that is endowed with the weak$^*$
topology. A state $s$ on $E$ is extremal if and only if $\mu_s$ in
{\rm (5.2)} is extremal. In such a case, $\mu_s = \delta_s.$
\end{corollary}

\begin{proof}
Due to Theorem \ref{th:7.5'} and (5.2),  the mapping $s\mapsto
\mu_s$ is affine and injective. If $\mu$ is a regular Borel
probability measure with $\mu(\partial_e \mathcal S(E))=1,$ then
$\mu $ defines via (5.2) some state, $s,$ on $E.$ Hence, the mapping
is surjective. The continuity follows from \cite[Thm
II.4.1(iii)]{Alf}.

It is clear that $F$ is a face. Since the set $\partial_e \mathcal
S(E)$ is closed, due to \cite[Prop 5.25]{Goo}, $F$ is closed.

At any rate, every Dirac measure $\delta_s$ also with
$\delta_s(\partial_e \mathcal S(E))=1$ is always a regular Borel
probability measure. Equality (5.2) entails that $s$ has to be
extremal. Conversely, if $s$ is extremal, the uniqueness of $\mu_s$
yields that $\mu_s=\delta_s.$
\end{proof}

It is worthy to remark a note concerning formula (5.2) that if $\mu$
is any regular Borel measure, the formula (5.1) defines a state, say
$s_\mu,$ on $E.$ But if $\mu(\partial_e \mathcal S(E))<1,$ then for
$s_\mu$ there is another unique regular Borel probability measure
$\mu_0$ such that $\mu_0(\partial_e \mathcal S(E))=1$ and it
represents $s_\mu$ via (5.2).

\begin{corollary}\label{co:5.8}  Let $(G,u)$ be a unital po-group
satisfying {\rm (RDP)} and let $s$ be a state on it. Then there is a
unique maximal regular Borel probability measure $\mu_s \sim
\delta_s$ on $\mathcal B(\mathcal S(G,u))$ such that

$$ s(g) = \int_{\mathcal S(G,u)} \hat g(x) \dx \mu_s(x),\quad g \in
G. \eqno(5.3)$$

If, in addition, $(G,u)$ satisfies {\rm (RDP)$_2$},  there is a
unique regular Borel probability measure, $\mu_s,$ on $\mathcal
B(\mathcal S(G,u))$ such that $\mu_s(\partial_e \mathcal S(G,u))=1$
and
$$ s(g) = \int_{\partial_e \mathcal S(G,u)} \hat g(x) \dx \mu_s(x),\quad g \in
G.$$
\end{corollary}

\begin{proof} Due to Proposition \ref{pr:4.1},  the state spaces of
$E=\Gamma(G,u)$ and of $\Gamma(G,u)$ are affinely homeomorphic and
every state on $E$ can be uniquely extended to a state on $(G,u).$
The statements follow easily from Theorem \ref{th:7.3'}, (5.1) and
from Theorem \ref{th:7.5'} and (5.2), respectively. Indeed,  let
$s\in \mathcal S(G,u).$  We define a mapping $\theta: \mathcal
S(G,u)\to \mathcal S(E)$ defined by $\theta(s):= s|E.$  By
Proposition \ref{pr:4.1}, $\theta$ is an affine homeomorphism.

It is enough to assume $g \in G^+.$ Then $g = a_1+\cdots+a_n$ with
$a_1,\ldots,a_n \in E.$ Then $s(g)=s(a_1)+\cdots+s(a_n)=
\theta(s)(a_1)+\cdots+\theta(s)(a_n)= \hat a_1(\theta(s))+\cdots +
\hat a_n(\theta(s))= \hat g(s).$ Let $\nu_s$ be a unique regular
Borel measure defined on $\mathcal B(\mathcal S(E))$ such that
$\nu_s \sim \delta_{\theta(s)}$ and (5.1) holds.  If we set $\mu_s:=
\nu_s\circ \theta$, then $\mu_s$ is a unique regular measure on
$\mathcal S(G,u)$ such that $\mu_s \sim \delta_s.$  Then (5.1) gives

$$ s(g) = \sum_{i=1}^n \int_{\mathcal S(E)} \hat a_i(y)\dx
\nu_s(y)= \int_{\theta^{-1}(S(E))} \hat g(\theta(x))\dx
\nu_s(\theta(x)) =\int_{\mathcal S(G,u)} \hat g(x)\dx \mu_s(x).$$

\end{proof}

\section{Conclusion}

The states on pseudo effect algebras are analogous of probabilities
appearing in quantum measurements. In many situations, an effect
algebra or a pseudo effect algebra is an interval in a unital
po-group. Their state spaces are convex compact Hausdorff spaces
that are sometimes empty. If the state space is nonempty, then some
kind of the Riesz Decomposition Property allows us to show that the
state space is a Choquet simplex, Theorem \ref{th:4.2}, or even a
Bauer simplex, Theorem \ref{th:4.4}.

However, the state space of the crucial example of the Hilbert space
quantum mechanics, $\mathcal E(H),$ is not a simplex, the state
spaces of commutative C$^*$-algebra are simplices.

In Section 5, we have showed that if a state of a pseudo effect
algebra is with (RDP)$_1$ or with (RDP)$_2,$ then it can be
expressed as an integral of some continuous affine function through
a regular Borel probability measure, formulas (5.1) and (5.2), even
with some kind of uniqueness.

Formulas (5.1) and (5.2) are interesting also in other point of
view: According to de Finetti, a probability measure  is only  a
finitely additive measure, and by Kolmogorov \cite{Kol}, a
probability measure is assumed to be $\sigma$-additive. The
mentioned formulas show that there is a natural coexistence between
both approaches.

\end{document}